\documentclass[12pt]{article}
\usepackage{amsmath,amsthm,fullpage}
\usepackage{amssymb}
\usepackage{epsfig}
\usepackage{tikz}
\usepackage{graphicx}
\usepackage{url}
\usepackage{authblk}
\theoremstyle{plain}
\newtheorem{theorem}{Theorem}[section]
\newtheorem{corollary}[theorem]{Corollary}
\newtheorem{lemma}[theorem]{Lemma}

\theoremstyle{definition}
\newtheorem{definition}[theorem]{Definition}
\newtheorem{remark}[theorem]{Remark}

\newtheorem{problem}[theorem]{Problem}

\newcommand{\Aut}{\mathop{\mathrm{Aut}}}
\newcommand{\stab}{\mathop{\mathrm{stab}}}
\newcommand{\fix}{\mathop{\mathrm{fix}}}

\newcommand{\supp}{\mathop{\mathrm{supp}}}

\let\OLDthebibliography\thebibliography
\renewcommand\thebibliography[1]{
	\OLDthebibliography{#1}
	\setlength{\parskip}{0pt}
	\setlength{\itemsep}{0pt plus 0.3ex}
}

\begin{document}
	\setcounter{Maxaffil}{3}
	\title{On minimal graphs, fixing sets and base size sets for hamiltonian groups}
	\author[ ]{Kirti Sahu\thanks{kirtisahu157@gmail.com}}
	\author[ ]{Ranjit Mehatari\thanks{ranjitmehatari@gmail.com, mehatarir@nitrkl.ac.in}}
	\affil[ ]{Department of Mathematics,}
	\affil[ ]{National Institute of Technology Rourkela,}
	\affil[ ]{Rourkela - 769008, India}
	\maketitle
	\begin{abstract}
		A finite non-abelian group $H$ is hamiltonian if all of its subgroups are normal. We compute the minimal orders of graphs having a hamiltonian group as their automorphism group. Later, we determine the fixing sets and base size sets corresponding to finite hamiltonian groups. As a consequence, we obtain that if $H$ is a hamiltonian group, then the base size set of $H$ is equal to its fixing set.
	\end{abstract}
	
	
	\section{Introduction}
	Suppose $\Gamma$ be a simple finite graph having vertex set $V(\Gamma)$ and edge set $E(\Gamma)$. Define an automorphism of graph $\Gamma$ as an adjacency-preserving permutation from its vertex set to itself. Then the set consisting of all automorphisms of $\Gamma$, represented by $\Aut \Gamma$, forms a group known as automorphism group of $\Gamma$. In 1936, K\"{o}nig asked a question in his book: when is it possible for a given group be expressed as the automorphism group of some graph? In response to this question, Frucht \cite{rf} proved that for every finite group $G$, there exist a graph $\Gamma$ such that the automorphism group of $\Gamma$ is isomorphic to $G$. Later, he proved that if $\Gamma$ is a cubic graph \cite{rf1} then the solution to the above question is also possible. In 1957, Sabidussi \cite{gs} proved the question with some extra properties of graphs, such as prescribed regularity, prescribed chromatic number, or vertex-connectivity. However, for a given group $G$, we have infinite number of graphs having group $G$ as their full automorphism group. Considering $G$ to be a finite group, we  define $\alpha(G)$ as the minimum number of vertices among all graphs $\Gamma$ such that $\Aut \Gamma \cong G$. Babai proved in \cite{lb} that if $G$ is a finite group other than $C_3$, $C_4$, and $C_5$, then $\alpha(G)\leq 2|G|$. The exact value of $\alpha(G)$ for various groups has been determined completely: finite abelian groups \cite{wca}, symmetric groups \cite{lv}, alternating groups $A_{n}$ for degree $n \geq 13$ \cite{mw}, hyperoctahedral groups \cite{adg}, dihedral groups \cite{lsc,lc,gh,dj}, generalized quaternion groups \cite{lscg}, quasi-abelian groups \cite{lz}, quasi-dihedral groups \cite{lz}, and modular $p$-groups \cite{sm}.\medskip
	
	Throughout this paper, we consider a family of groups called hamiltonian groups. A non-abelian group $H$ is said to be hamiltonian if every subgroup of $H$ is normal. A nice characterization of hamiltonian groups is given as follows.
	\begin{theorem} \cite{mha}
		A hamiltonian group is the direct product of a quaternion group with an abelian group in which every element is of finite odd order and an abelian group of exponent two.
	\end{theorem}
	Arlinghaus \cite{wca} evaluated the value of $\alpha(A)$, for $A$ to be a finite abelian group. His results are too lengthy to recall here. We use a few of these results to determine $\alpha(H)$, where $H$ is a hamiltonian group.

	Let $G$ be a group that acts on set $X$, and $x\in X$. Define $\stab(x) = \{g \in G : g(x) = x\}$ as the stabilizer of $x$. The set-wise stabilizer of set $Y \subseteq X$ is $\stab(Y) = \{g \in G : g(y) = y, \forall y \in Y\}$. For a graph $\Gamma$, the fixing number (equivalently, the determining number defined in \cite{ldb}) is defined as the minimum cardinality of vertex subset $S$ of $V(\Gamma)$ such that $\stab(S)$ is trivial, and $S$ is called a fixing set of $\Gamma$. The study of fixing sets of different graphs families have been studied by various authors. While the fixing numbers for certain cases of Kneser graphs have been studied in \cite{ldb,cggmp} and for Fullerene graphs in \cite{kam}, the fixing numbers for trees and complete graphs have been characterized in \cite{fhed}. The concept of fixing the number of graphs has been extended to finite groups by Gibbons and Laison \cite{jdl}. For a finite group $G$, the fixing set ($\fix(G)$) is defined as the set consisting all fixing numbers of graphs having group $G$ as their full automorphism group. Researchers have completely determined the fixing sets of finite abelian groups \cite{jdl}, generalized quaternion groups \cite{lklgchh}, dihedral groups \cite{lkla}, quasi-abelian groups \cite{ummul}, quasi-dihedral groups \cite{ummul}, and modular $p$-groups \cite{sm}. For finite abelian groups we have the following theorem.
	\begin{theorem}\cite{jdl}
		\label{fixing_abelian}
		Suppose for a finite abelian group $A$, the number of elementary divisors be $d$. Then $\fix(A)=\{1,2,\ldots,d\}$.
	\end{theorem}
	In \cite{jdl}, Gibbons and Laison proved the following lemma:
	\begin{lemma}
		\cite{jdl}\label{fixing_lm1}
		Suppose $G_{1}$ and $G_{2}$ are two non-trivial finite groups, then the set $\fix(G_{1}) + \fix(G_{2}) \subseteq \fix(G_{1} \times G_{2})$, where the sum $\fix(G_{1}) + \fix(G_{2})$ is defined as $\{a + b 
		\mid a \in \fix(G_{1}), b \in \fix(G_{2})\}$.
	\end{lemma}
	Further,they posed the following open problem:	\begin{problem}\cite{jdl}
		\label{fix_op1}
		Let $G_1$ and $G_2$ be any two finite non-trivial groups. Is it true that $\fix(G_1) +
		\fix(G_2) = \fix(G_1 \times G_2)\setminus\{1\}$?
	\end{problem}
	The above open problem is certainly true if both $G_1$ and $G_2$ are abelian. We solve Problem \ref{fix_op1} partially by considering $\fix(G_1)=\{1,2,\ldots,d\}$ and $\fix(G_2)=\{1\}$. As a consequence of the above result, we determine the fixing set for hamiltonian groups. 
	
	Let $Sym(X)$ denote the symmetric group on the set $X$. For a permutation representation $\tau : G \to Sym(X)$, base of $\tau$ is a subset $B \subseteq X$ such that the $\stab(B) = \{g \in G : g(b) = b, \forall b \in B\}$ is trivial. The base size $b_{\tau}(G)$ is defined as the cardinality of smallest base of $\tau$. 
	 
	 	\begin{definition}
	 	The base size set is defined as the set of all base sizes of all faithful representations of $G$ on finite sets, that is, $\mathcal{B}(G) = \{ b_{\tau}(G) : \tau \in \mathcal{T}(G)\}$, where $\mathcal{T}(G)$ is set of all finite faithful representations of $G$.
	 \end{definition}
	 
	 	Suppose $G$ be a finite group acting faithfully on a set $X$ with representation $\tau$. Let $B = \{b_{1}, b_{2}, \ldots,b_{n}\}$ be the minimal base of $\tau$. Then we have a chain of subgroup $G > \stab(b_{1}) > \stab(b_{1},b_{2}) > \ldots > 1$. Define the length of $G$ as the size of longest chain of subgroups of $G$ (not counting 1). Then in \cite{cameron} it has been proved that the length of $G$ is at most the number of prime factors of $|G|$, with counting multiplicities. We have following result.
	 \begin{lemma}
	 	\label{base_len}
	 Let $G$ be a finite group of length $n$. Then $\fix(G) \subseteq \mathcal{B}(G) \subseteq \{1,2,\ldots,n\}$.
	 \end{lemma}
	 Above lemma shows that every finite group $G$ has finite fixing set and finite base size set. In $\cite{laison}$, Laison \emph{et al.} compared the fixing sets and base size sets of finite abelian groups and dihedral groups of order $2p^{k}$, $4p^{k}$ and $2pq$, where $p,q$ are odd primes. In this article, we examine the relation between the base size set and the fixing set of a hamiltonian group. Our results provide a partial answer to the open question posed in \cite{laison}: given the base size sets of two groups, what is the base size set of their direct product?

	The structure of this article is as follows: In Section 2, we discuss several useful results and the notions that we use later on to prove our main results. In Section 3, we construct vertex-minimal graphs for a finite hamiltonian group $H$ and determine $\alpha(H)$. In Section 4, we provide a partial solution to Problem \ref{fix_op1} and compute $\fix(H)$. Finally in Section 5, we establish the relation between base size set and fixing set of hamiltonian group.

	\section{Preliminaries}

	This section presents several key concepts and preliminary results that serve as essential tools for proving the main results of this paper. Throughout this paper, all graphs and grops considered are assumed to be finite, unless explicitly stated otherwise.  
	
	This paper is devoted to the computation of $\alpha(H)$. Throughout the work, we consider $H$ as a permutation group. Let $G$ be a permutation group acting on a set $X$ with $n$ symbols, and $x\in X$. Then the set $O(x) = \{g(x) : g \in G\}$ is defined as the orbit of $x$. 
	
	Let $\Gamma$ be a graph with the vertex set $V(\Gamma)$ and the edge set $E(\Gamma)$. If $G$ is a subgroup of the symmetric group $S_{V(\Gamma)}$, then the Edge Orbit of an edge $[v_1,v_2]\in E(\Gamma)$ is defined as $$O\{v_{1},v_{2}\}= \Big\{[g(v_{1}),g(v_{2})] : g \in G\Big\}.$$
	
	Also, we define the support of a permutation $g \in G$ as $\supp(g) = \{x\in X : g(x) \neq x\}$. 
	
	Let $n$ be positive integer with $n \geq 3$. Define the generalized quaternion group of order $2^{n}$ as,
	
	\begin{equation}
		\label{quaternion_equation}
	Q_{2^{n}} = \langle r,f : r^{2^{n-1}} = 1 = f^{4}, frf^{-1} = r^{-1}, r^{2^{n-2}} = f^{2} \rangle
	\end{equation}
	Observe that every element in $Q_{2^{n}} \setminus \langle r \rangle$ is of order four and $r^{2^{n-2}} = f^{2}$ is the only element of order $2$.
	The following lemma describes the subgroup representation of quaternion group $Q_{8}$ in symmetric group, which is as follows.
	\begin{lemma}
		\label{ham_lem1}
		For any integer $k < 8$, the quaternion group $Q_{8}$ cannot be embedded as a subgroup of the symmetric group $S_{k}$.
	\end{lemma}
	
	\begin{proof}
		Since $Q_{8}$ has elements of order $4$, it cannot be embedded into symmetric group $S_{k}$ with $k \leq 3$. Suppose $Q_{8}$ is isomorphic to a subgroup of $S_{k}$ where $4 \leq k \leq 7$. Then the cycle decomposition of a $4$-order element in $S_{k}$ is either a $4$-cycle or a product of a $4$-cycle and a $2$-cycle. Let the generators of $Q_{8}$ be $\sigma$ and $\tau$, with the cycle decomposition of $\sigma$ as $(a_{1}, a_{2}, a_{3}, a_{4})\sigma_{1}$, where $\sigma_{1}$ is a cycle of length at most $2$. The group $Q_{8}$ has 6 elements of order 4 and a unique involution. Thus, $\sigma^{2} = \tau^{2} = (a_{1}, a_{3})(a_{2}, a_{4}) \sigma_{1}^{2}$, therefore $\tau = (a_{1}, a_{i}, a_{3}, a_{j})$ for some $i,j \in  \{2,4\}$, which is a contradiction.
	\end{proof}
	
	Now, using Cayley's theorem, we conclude that 8 is the smallest positive integer such that $Q_8$ is isomorphic to a subgroup of $S_8$; and one can get the generators of $Q_8$ in $S_8$ as $\sigma=(1,2,3,4)(5,6,7,8)$, $\tau=(1,7,3,5)(2,6,4,8)$. 
	
	For a finite group $G$, let $\mu(G)$ denote the smallest positive integer $n$ for which $G$ is isomorphic to a subgroup of $S_n$. In 1971, Johnson \cite{jh} first considered the problem: what is the minimal permutation representation of a finite group? Later, Wright \cite{wd} considered the problem for direct products.
	
	Now for any finite group $G$, one can observe that $\mu(G)\leq\alpha(G)$, where the equality do not hold in general. In fact, there are permutation groups that are not an automorphism group of some graph, for example the subgroup $\{1, (12)(34),(13)(24),(14)(23)\}$ of $S_4$ is not an automorphism group of any graph. Also for the quaternion group $Q_8$, we have $$8=\mu(Q_8)<\alpha(Q_8)=16.$$
	
	Using, Theorem of \cite{wd}, we have the following conclusions
	\begin{itemize}
		\item
		$\mu(C_2^k)=2k$.
		\item
		$\mu(Q_8\times C_2^k)=8+2k$.
	\end{itemize}   
	
	Thus, for the elementary abelian 2-group $C_2^k$, we have $\mu(C_2^k)=\alpha(C_2^k)=2k$.

	Let $G$ be a group and $S$ be the generating set of $G$. With the vertex set $G$, the Cayley digraph $C(G,S)$ is defined as the directed and edge-labelled multigraph such that there is a directed edge from group element $g_{1}$ to $g_{2}$ labelled with the generator $s \in S$ if and only if $sg_{1} = g_{2}$. Now define the Frucht Graph $F(G,S)$ from $C(G,S)$ by replacing each directed and labelled edge with a graph gadget such that $\Aut(F(G,S)) = \Aut(C(G,S))$. The Frucht graph is discussed in detail in \cite{bw}. The following lemma follows from the exposition in \cite{bw}. 
	\begin{lemma}
		\label{ham_lm2}
		Let $G$ be a group, and $S$ be its generating set.
		Then $\Aut(C(G,S)) = G$ and $\Aut(F(G,S)) = G$. Furthermore, the fixing number of $C(G,S)$ or $F(G,S)$ is $1$.
	\end{lemma}
	
	The above lemma ensures that for every vertex $v$ of a Frucht graph or a Cayley digraph, $\stab(v)$ is trivial. Consequently, for a finite non-trivial group $G$, $1 \in \fix(G)$. 	
	
	\section{Value of $\alpha(H)$}
	This section is devoted to computing the value of $\alpha(H)$. From the discussions in the previous section, graph having automorphism group isomorphic to the hamiltonian group $H\cong Q_8\times A$ must have at least $8+\mu(A)$ vertices. In the following theorem, we find $\alpha(H)$, where $H$ is a hamiltonian 2-group.
\begin{figure}[h]
	\centering
	\includegraphics[width=0.75\linewidth]{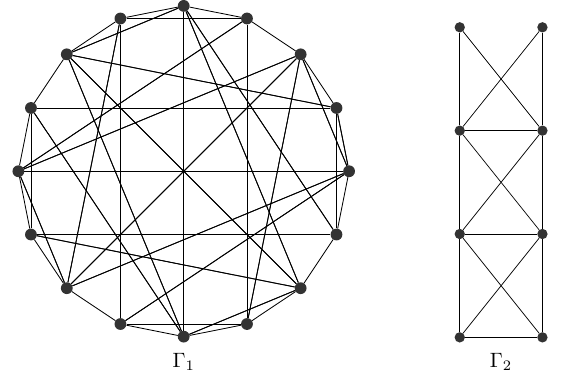}
	\caption{Vertex-minimal graphs with $\Aut \Gamma_{1} = Q_{8}$ and $\Aut \Gamma_{2} = C_{2}^{4}$.}
	\label{fig:mainfig1}
\end{figure}
	\begin{theorem}
		\label{ham_th1}
		Let $k$ be any positive integer. Then  $\alpha(Q_{8} \times C_{2}^{k}) = 16 + 2k$.
	\end{theorem}
	
	\begin{proof}
		Let $H=Q_8\times C_2^k$. First we observe that $\alpha(H)\geq\mu(H)=8+2k$. Again, since $\alpha(Q_8)=16$ and $\alpha(C_2^k)=2k$, it follows that $8 + 2k \leq \alpha(H) \leq 16 + 2k$. We will show that there does not exist any graph with vertex less than $16+2k$ with hamiltonian $2$-group symmetry. It is sufficient to show that $\alpha(H)\neq 8+2k, 10+2k, 12+2k, 14+2k$.\\
		First, if possible, let $\Gamma_{1}$ be a graph with $8+2k$ vertices with the automorphism group $H$. Without loss of generality we take
		\begin{align*}
			\Aut \Gamma_{1} &= \langle \sigma, \tau, \delta_{1}, \delta_{2}, \ldots, \delta_{k} \rangle \\
			&= \langle (u_{1}, u_{2}, u_{3}, u_{4})(v_{1}, v_{2}, v_{3}, v_{4}), (u_{1}, v_{3}, u_{3}, v_{1})(u_{2}, v_{2}, u_{4}, v_{4}), (a_{11}, a_{12}), \ldots, (a_{k1}, a_{k2})\rangle .
		\end{align*}
		
		Then this group has $2(2^{k} -1)+1$ elements involutions. Take $S$ be the set of all involutions. We claim that $\gamma = (u_{1}, u_{3})(u_{2}, u_{4})(a_{11}, a_{12})(a_{21}, a_{22})\ldots(a_{k1}, a_{k2})$ is also an element of $\Aut \Gamma_{1} \setminus S$, which is absurd.\\
		For $\Gamma_1$, we have the following possible edge orbits:
		\begin{align*}
			&\{[u_{i}, u_{j}]\},\ 
			\{[u_{i}, v_{j}]\},\ 
			\{[v_{i}, v_{j}]\},\\
			&\{[u_{i}, a_{l1}]\mid \forall \ l =1,2,\ldots,k\}, \
			\{[u_{i}, a_{l2}]\mid \forall \ l =1,2,\ldots,k\},\\ 
			& \{[v_{i}, a_{l1}]\mid \forall \ l =1,2,\ldots,k\},\	
			\{[v_{i}, a_{l2}]\mid \forall \ l =1,2,\ldots,k\},\\
			& \{[a_{l1}, a_{m2}]\mid \forall \ l,m =1,2,\ldots,k\}.
		\end{align*}	 
		Now, observe that
		\begin{align*}
			&\gamma([u_{i}, u_{j}]) = \sigma^{2}([u_{i}, u_{j}]),\\
			&\gamma([u_{i}, v_{j}]) = \sigma^{i+j}\tau ([u_{i}, v_{j}]),\\
			&\gamma([u_{i}, a_{l1}]) = \sigma^{2} \delta_{1} \delta_{2} \ldots \delta_{k} ([u_{i}, a_{l1}]) = \gamma([u_{i}, a_{l2}]),\\
			&\gamma([v_{i}, v_{j}]) = e,\\
			&\gamma([v_{i}, a_{l1}]) = \delta_{1} \delta_{2} \ldots \delta_{k} ([v_{i}, a_{l1}]) = \gamma([v_{i}, a_{l2}]),\\
			&\gamma([a_{l1}, a_{m2}]) = \delta_{1} \delta_{2} \ldots \delta_{k} ([a_{l1}, a_{m2}]).
		\end{align*}
		From the above calculations it follows that, for each edge $[u,v]\in E(\Gamma_1)$, an automorphism $\beta$ of $\Gamma_1$ exists such that $\gamma([u,v]) = \beta ([u,v])$. Which implies $\gamma\in \Aut \Gamma_1$. Therefore, $\alpha(H)\neq 8+2k$.


		Next we suppose there exist a graph $\Gamma_{2}$ with $\Aut \Gamma_{2} = Q_{8} \times C_{2}^{k}$ with $8+2k+2$ number of vertices. Two cases arise:
		\begin{itemize}
			\item [(a)] If $\sigma = (u_{1}, u_{2}, u_{3}, u_{4})(v_{1}, v_{2}, v_{3}, v_{4})(b_{1}, b_{2})$ and $\delta_{l} = (a_{l1}, a_{l2})$ for each $l \in \{1,2,\ldots,k\}$.
			We claim  that  $\gamma = (u_{1}, u_{3}) (u_{2}, u_{4})(a_{11}, a_{12})\ldots(a_{k1}, a_{k2}) \in \Aut \Gamma_{2} \setminus S$, which is a contradiction.\\
			We have following possible edge orbits:
			\begin{align*}
				& \{[u_{i}, u_{j}]\},\
				\{[u_{i}, v_{j}]\},\
				\{[v_{i}, v_{j}]\},\\
				&\{[u_{i}, b_{1}]\},\
				\{[v_{i}, b_{1}]\},\
				\{[b_{1}, b_{2}]\},\\
				&\{[u_{i}, a_{l1}], \forall \quad l =1,2,\ldots,k\},\
				\{[u_{i}, a_{l2}], \forall \quad l =1,2,\ldots,k\},\\	
				&\{[v_{i}, a_{l1}], \forall \quad l =1,2,\ldots,k\},\
				\{[v_{i}, a_{l2}], \forall \quad l =1,2,\ldots,k\},\\
				&\{[a_{l1}, a_{m2}], \forall \quad l,m =1,2,\ldots,k\},\
				\{[b_{1}, a_{l1}], \forall \quad l =1,2,\ldots,k\},\\
				&\{[b_{1}, a_{l2}], \forall \quad l =1,2,\ldots,k\}.		
			\end{align*}
			
			Observe that 
			\begin{align*}
				&\gamma([u_{i}, u_{j}]) = \sigma^{2}([u_{i}, u_{j}]) = \gamma([u_{i}, b_{1}]) = \gamma([v_{i}, b_{1}]) , \\
				& \gamma([u_{i}, v_{j}]) = \sigma^{i+j}\tau ([u_{i}, v_{j}]), \\
				&\gamma([u_{i}, a_{l1}]) = \sigma^{2} \delta_{1} \delta_{2} \ldots \delta_{k} ([u_{i}, a_{l1}]) = \gamma([u_{i}, a_{l2}]),\\
				&\gamma([v_{i}, v_{j}]) = e = \gamma([b_{1}, b_{2}]),\\
				&\gamma([v_{i}, a_{l1}]) = \delta_{1} \delta_{2} \ldots \delta_{k} ([v_{i}, a_{l1}]) = \gamma([v_{i}, a_{l2}]), \\
				&\gamma([a_{l1}, a_{m2}]) = \delta_{1} \delta_{2} \ldots \delta_{k} ([a_{l1}, a_{m2}]) = \gamma([b_{1}, a_{l1}]) = \gamma([b_{1}, a_{l2}]) .\\
			\end{align*}
			Hence $\gamma \in \Aut \Gamma_{2} \setminus S$, which is a contradiction.
			
			\item [(b)]   If $\sigma = (u_{1}, u_{2}, u_{3}, u_{4})(v_{1}, v_{2}, v_{3}, v_{4})$ and $\delta_{1} = (a_{11}, a_{12})(b_{1}, b_{2})$  and for each $l \in \{2,\ldots,k\}$ $\delta_{l} = (a_{l1}, a_{l2})$. Then in a similar way as above, we observe that\\
			$\gamma = (u_{1}, u_{3}) (u_{2}, u_{4})(a_{11}, a_{12})\ldots(a_{k1}, a_{k2})(b_{1}, b_{2}) \in \Aut \Gamma_{2} \setminus S$, a contradiction.
		\end{itemize}

		In similar ways, we observe that there does not exist any graph that has automorphism group isomorphic to $Q_{8} \times C_{2}^{k}$ with an $8+2k+4$ or $8+2k+6$ number of vertices.    
	\end{proof}

	\begin{theorem}
		\label{ham_main1}
		Let $H$ be a hamiltonian group of the form $Q_{8} \times A$, where $A$ is some periodic abelian group that has no element of order $4$. Then $\alpha(Q_{8} \times A) = 16 + \alpha(A)$.
	\end{theorem}	
	\begin{proof}
		From the disjoint union of vertex minimal graphs of $Q_{8}$ and $A$, we have $\alpha(H) \leq 16 + \alpha(A)$. Suppose $Q_{8}$ be generated by $\sigma$ and $\tau$. Let $1 \neq a \in H$ be the generator of $A$ such that $a \notin \langle \sigma, \tau \rangle$. Our claim is that with less than $16 +\alpha(A)$ number of vertices, there exist no graph with $H$-symmetry, that is, $\supp(\sigma) \cap \supp(a) = \phi$.\medskip

		Towards a contradiction, suppose there exists a graph $\Gamma$ such that $\Aut (\Gamma) \cong Q_{8} \times A$ with $|V(\Gamma)| < 16 + \alpha(A)$. Two cases arise here: \medskip\\		
		\textbf{Case I:} If order of element $a$ is even. Then, clearly it follows from Theorem \ref{ham_th1} it follows  that $\supp(\sigma) \cap \supp(a) = \phi$.\medskip\\
		\textbf{Case II:} If order of element $a$ is odd. Let $A = \mathbb{Z}_{p_{1}^{\alpha_{1}}} \times \mathbb{Z}_{p_{2}^{\alpha_{2}}} \times \ldots \times \mathbb{Z}_{p_{n}^{\alpha_{n}}} $ be an abelian group of odd order, where $p_{i}$'s are odd and $\alpha_{i} \geq 1$ for all $i = 1,2, \ldots, n$. Let $\Gamma_{1}$ be a vertex-minimal graph of $A$. Then the cycle decomposition of the generators of \(\Aut \Gamma_{1}\) will follow the structure of the vertex- minimal graph. Suppose $\supp(a) \cap \supp(\sigma) \neq \phi$. Then, we have $\sigma^{2} a \neq a \sigma^{2}$, a contradiction.
	\end{proof}
	
	For $G_1=Q_8$ and an abelian group $G_2$ without any element of order $4$, we have proved that  
	$$\alpha(G_{1} \times G_{2}) = \alpha(G_{1}) + \alpha(G_{2}).$$ 
	However, the above equality does not hold for arbitrary non-trivial groups $G_1$ and $G_2$. For example (see \cite{wca} for details), $\alpha(C_3\times C_5)=\alpha(C_3)+\alpha(C_5)-3$. So the following obvious problem arises:
	\begin{problem}
		Classify all non-trivial finite groups $G_1$ and $G_2$, such that $\alpha(G_{1} \times G_{2}) = \alpha(G_{1}) + \alpha(G_{2}).$
	\end{problem}

	\section{Fixing set of hamiltonian group}
	In this section, we determine $\fix(H)$. Let $A$ be a finite abelian group with its unique factorization. Then elementary divisors of $A$ are defined as the set of all prime power orders of cyclic groups.
	
\begin{figure}[h]
	\centering
	\includegraphics[width=0.85\linewidth]{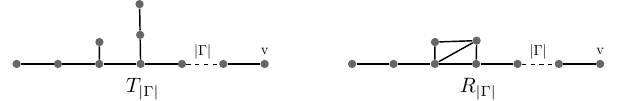}
	\caption{Graphs with identity automorphism}
	\label{fig:mainfig}
\end{figure}
	
	In the subsequent theorem, we demonstrate a natural extension of Theorem \ref{fixing_abelian}. We use a similar technique here. The theorem also provides a partial solution to Problem \ref{fix_op1}. For a positive integer $k$, let $I_k$ denote the set $\{1,2,\ldots, k\}$.
	\begin{theorem}
		Let $G_{1}$ and $G_2$ be two finite groups such that $\fix(G_{1})=\{1\}$ and $\fix(G_2)=I_d$, then $$\fix(G_1 \times G_2)=I_{d+1}.$$
	\end{theorem}	
	\begin{proof}
		Let $G= G_{1} \times G_{2}$. By Lemma \ref{fixing_lm1} and \ref{ham_lm2}, it follows that $$I_{d+1} \subseteq \fix(G_{1} \times G_{2}).$$
		
		Consider a finite graph $\Gamma$ having $\Aut \Gamma = G_{1} \times G_{2}$. Let us consider the subgroup $G_{1}$ of $G$. Since $\fix(G_{1}) = \{1\}$, there exist a vertex $v_{1}$ such that $G_{1}$ acts regularly on $O(v_{1})$.
		
		
		Consider the vertex $v_1$ and let $\Gamma_{1}$ be the connected component of $\Gamma$ that contains $v_{1}$. Then either $\Gamma_{1}$ is a tree or it contains a cycle. If $\Gamma_{1}$ is a tree, then the graph $\Gamma^{'}$ is constructed by $\Gamma$ by joining the graph $R_{|\Gamma|}$ as illustrated in Figure 2 to $\Gamma$. If $\Gamma_1$ is not a tree, then we construct $\Gamma'$ from $\Gamma$ by attaching the graph $T_{|\Gamma|}$ to the vertex $v_{1}$ belogs to $\Gamma$. Let us denote the subgraph $R_{|\Gamma|}$ or $T_{|\Gamma|}$ in $\Gamma^{'}$ by $\Gamma_{2}$. Our claim is that $\Aut(\Gamma^{'})$ is a subgroup of $G_{2}$. To prove our claim, we first show that $\Gamma^{'}$ has no additional automorphism that $\Gamma$ does not. If possible, let $\tau$ be an automorphism of $\Gamma^{'}$ that is not an automorphism of $\Gamma$. Then there must exist a vertex $v$ in $\Gamma_2$ such that $\tau(v)\neq v$. Since $\Aut \Gamma_2=\{1\}$, every vertex in $\Gamma_2$ must be in $\supp (\tau)$. Again, since $\Gamma'$ contains more vertices than $\Gamma$, the image of a vertex of $\Gamma_{2}$ under the permutation $\tau$ must be another vertex of $\Gamma_{2}$. Therefore $\tau$ is completely determined as a permutation on the vertices $\Gamma_{1} \cup \Gamma_{2}$. In fact, $\tau$ is a reflection of $\Gamma_{1} \cup \Gamma_{2}$ about some vertex of $\Gamma_{2}$. Which seems impossible, since we have constructed $\Gamma'$ in a way that $\Gamma_{2}$ contains a cycle if and only if $\Gamma_{1}$ does not.
		
		On the other hand, we can not extend any automorphism of $\Gamma$ as an automorphism of $\Gamma'$, as in the graph $\Gamma'$, the vertex $v_1$ has a larger degree than all the vertices of $O(v_{1})$; therefore, $v_1$ cannot be mapped to any vertex of $O(v_{1})$. And this proves our claim.
		
		Hence, the fixing number of $\Gamma'$ is bounded above by $d$. Let $S$ be a fixing set of $\Gamma'$ where $|S|\in I_d$.	Let $S' = S \cup \{v_{1}\}$. Then $|S'|\in I_{d+1}$ and the set $S'$ serves as a fixing set of $\Gamma$. Therefore $\fix(G) = \{1,2,\ldots,d+1\}$. \end{proof} 
	In \cite{lklgchh}, Christina and Lauderdale proved that $\fix(Q_8)=\{1\}$. Thus, by using Theorem \ref{fixing_abelian}, we obtain the following corollary that describes fixing set of a hamiltonian group.

	\begin{corollary}
		\label{ham_main3}
		Let $H$ be a hamiltonian group of the form $Q_{8} \times A$, where $A$ is periodic abelian group that has no element of order $4$. Let $d$ denote the number of elementary divisors of $A$. Then the set $\fix(H) = \{1,2,\ldots,d+1\}$.
	\end{corollary}	
	
	\section{Base size set of hamiltonian group}
In this section, we first discuss the base size sets of the generalized quaternion groups, and then using the results for the quaternion group, we study the base size sets of hamiltonian groups. In \cite{lklgchh}, authors studied the vertex-minimal graph having automorphism group $Q_{2^{n}}$. They stated a result for the cycle decomposition of $Q_{2^{n}}$. 
	\begin{lemma} \cite{lklgchh}
		\label{cycle_quaternion}
		Consider $Q_{2^{n}}$ as a permutation group, and suppose $r$ and $f$ are the generators of $Q_{2^{n}}$ as defined in Equation \ref{quaternion_equation}. Then the cycle decomposition of $r$ comprises at least two $2^{n-1}$-cycles and the cycle decomposition of $f$ comprises at least $2^{n-2}$ $4$-cycles.
	\end{lemma}
	
Using above Lemma \ref{cycle_quaternion}, authors in \cite{lklgchh} studied the orbit size of an element in $Q_{2^{n}}$. The result is as follows.
	\begin{lemma} \cite{lklgchh}
		\label{cycle_quaternion1}
		Consider $Q_{2^{n}}$ as a permutation group, and suppose $r$ and $f$ are the generators of $Q_{2^{n}}$ as defined in Equation \ref{quaternion_equation}. There exists an element $a$ in  $\supp(r)$ such that $|O(a)| = 2^{n}$.
	\end{lemma}
	Thus using Lemma \ref{cycle_quaternion} and Lemma
	\ref{cycle_quaternion1}, a result follows immediately.
	
	\begin{theorem}
		\label{base_quaternion}
		Let $n \geq 3$ be positive integer. Then $\mathcal{B}(Q_{2^{n}}) = \{1\}$.
	\end{theorem} 
	\begin{proof}
		Suppose $\tau$ be faithful representation of $Q_{2^{n}}$ on set $S$. Let $r$ and $f$ be generators of $Q_{2^{n}}$. Using Orbit-Stabilizer Theorem and Lemma \ref{cycle_quaternion1}, we have
		\begin{equation*}
			|Q_{2^{n}}| = |O(a)|\cdot |\stab(a)| = 2^{n} \cdot |\stab(a)|.
		\end{equation*} 
		It implies that $|\stab(a)| = 1$. Hence $\mathcal{B}(Q_{2^{n}}) = \{1\}$.
	\end{proof}
	Thus we show that the base size sets and the fixing sets of generalized quaternion groups are equal.
	\begin{remark}
		\label{stab_rem}
		For a finite group $G$, let $\tau$ be a faithful representation on set $S$. Let $g \in G$ be an element of order $p^{k}$, then $\tau(g)$ is a permutation of order $p^{k}$. Thus the cycle decomposition of $\tau(g)$ has atleast one cycle of length $p^{k}$. Then there exists $x \in S$ suct that $|O(x)| \geq p^{k}$. Thus $|\stab(x)| \leq |G| /p^{k}$.
	\end{remark}
	Using Lemma \ref{base_len} and Corollary \ref{ham_main3}, for hamiltonian group $H$, we have $\{1,2,\ldots,d+1\} =\fix(H) \subseteq \mathcal{B}(H)$. However, we give a complete constructive proof to the following main result for finite hamiltonian group.

	\begin{theorem}
		\label{base_hamiltonian}
		Let $H$ be a hamiltonian group of the form $Q_{8} \times \mathcal{A}$, where $\mathcal{A}$ is some periodic abelian group that has no element of order $4$. Let $d$ denote the number of elementary divisors of $\mathcal{A}$. Then $\mathcal{B}(H) = \{1,2,\ldots,d+1\}$. 
	\end{theorem}
	\begin{proof}
	Let $Q_{8}$ and $\mathcal{A}$ be permutation groups on disjoint sets X and Y respectively. Suppose the direct product $Q_{8} \times \mathcal{A}$ has a permutation representation $\tau : Q_{8} \times \mathcal{A} \to  Sym(X \cup Y)$. Then $b_{\tau}(Q_{8} \times \mathcal{A}) = b_{\tau}(Q_{8}) + b_{\tau}(\mathcal{A})$. Suppose $S$ is a generating set of $Q_{8} \times \mathcal{A}$. Then by Lemma \ref{ham_lm2}, we have a Frucht graph $F(H,S)$ such that its automorphism group is $H$ itself. Thus there exist a vertex $v$  of $F(H,S)$ such that $\stab(v)$ is trivial. Therefore, $1 \in \mathcal{B}(Q_{8} \times \mathcal{A})$. Hence we conclude that $\{1,2,\ldots,d+1\} \subseteq \mathcal{B}(H)$.\medskip
		
		We now prove that $\mathcal{B}(H) \subseteq \{1,2,\ldots,d+1\}$. For this, we use the induction hypothesis on number of elementary divisors $d$ of the abelian group $\mathcal{A}$. First, consider $d =1$, i.e., $H = Q_{8} \times \mathbb{Z}_{p^{k}}$. For a finite set $S$, suppose $\tau : H \to Sym(S)$ is a faithful representation. Then for an element $g \in H$ of order $p^{k}$, there exists an element $x_{1} \in S$ such that $|O(x_{1})| \geq p^{k}$ and $|\stab(x_{1})| \leq 8$. This means $\stab(x_{1})$ is isomorphic to one of the groups: identity group, $\mathbb{Z}_{2}$, $\mathbb{Z}_{4}$ or $Q_{8}$. \medskip
			
			If $\stab(x_{1})$ is trivial,  then $\mathcal{B}(H) = \{1\}$.\medskip
			
			If $\stab(x_{1})$ is isomorphic to $\mathbb{Z}_{2}$ or $\mathbb{Z}_{4}$, then using Remark \ref{stab_rem}, there exists an element $x_{2} \in S$ from the induced representation of $\stab(x_{1})$ such that $\stab(x_{1},x_{2}) = \{1\}$.\medskip 
				
			Finally, if $\stab(x_{1}) \cong Q_{8}$, then by Theorem \ref{base_quaternion}, there exists an element $x_{2} \in S$ from the induced representation of $\stab(x_{1})$ such that $\stab(x_{1},x_{2}) = \{1\}$. Hence $\mathcal{B}(H) = \{1,2\}$.\medskip
		
		Now assume that the statement holds for all $l < d$, i.e., if $\mathcal{A}$ has $l$ elementary divisors, then the base size set for $H$ satisfies $\mathcal{B}(H) \subseteq \{1,2,\ldots,l+1\}$. Let $\mathcal{A}$ has $d$ elementary divisors. Then $H$ can be represented as
		\[ H \cong Q_{8} \times \mathbb{Z}_{p_{1}^{k_{1}}} \times \mathbb{Z}_{p_{2}^{k_{2}}} \times \ldots \times \mathbb{Z}_{p_{d}^{k_{d}}}.\]
		Now for each elementary divisor of $\mathcal{A}$, there is an element $g_{i}$ of $H$ of order $p_{i}^{k_{i}}$.\medskip
		
		Consider the subgroup $H_{1} \cong \mathbb{Z}_{p_{1}^{k_{1}}}$ of $H$ with a generator $g_{1}$. Let $\tau$ be faithful representation of $H$. Then the cycle decomposition of $\tau(g_{1})$ contains atleast one $p_{1}^{k_{1}}$-cycle, say $(x_{1},x_{2},\ldots,x_{p_{1}^{k_{1}}})$. Since for all $1 \leq r < p_{1}^{k_{1}}$, $g_{1}^{r}(x_{1}) \neq x_{1}$, it follows $H_{1} \cap \stab (x_{1}) = \{1\}$. Thus, by second and third isomorphism theorem, we have 
		\[\frac{\stab(x_{1})H_{1}}{H_{1}} \cong \frac{\stab(x_{1})}{H_{1} \cap \stab(x_{1})}\]
		or, \[\stab(x_{1}) \cong \frac{\stab(x_{1})H_{1}}{H_{1}} \leq \frac{H}{H_{1}}.\] 
		
		Therefore, it follows that \[\stab(x_{1}) \cong Q_{8} \times \mathbb{Z}_{p_{2}^{k_{2}}} \times \mathbb{Z}_{p_{3}^{k_{3}}} \times \ldots \times \mathbb{Z}_{p_{d}^{k_{d}}}.\] Hence $\stab(x_{1})$ has at most $d-1$ elementary divisors. Thus, by induction hypothesis, we have $\mathcal{B}(\stab(x_{1})) \subseteq \{1,2,\ldots,d\}$. If $B$ is a base of induced representation of $\stab(x_{1})$, then $B \cup \{x_{1}\}$ is a base of the permutation representation $\tau$ of $H$. Hence we prove that $\mathcal{B}(H) \subseteq \{1,2,\ldots,d+1\}$.    
		\end{proof}
	
	\begin{remark}By Corollary \ref{ham_main3} and Theorem \ref{base_hamiltonian} it follows immediately that $\mathcal{B}(H) = \fix(H)$.
	\end{remark}
	
	\subsection*{Data availability}
	No data were used in this article.
	\subsection*{Acknowledgements}
	Kirti Sahu is thankful to the CSIR for financial support (Grant No- 09/0983(16243)/2023-EMR-I).	
	
	\subsection*{Ethics declarations}
	\textbf{Conflict of interest:} The authors declare that they have no conflict of interest.

\end{document}